\documentclass[12pt,a4paper]{article}

\usepackage{amssymb,amsmath,amsfonts,amsthm,color,mathrsfs}
\usepackage{bbm}
\usepackage[marginal]{footmisc}

\title{\bf Sharp Li--Yau Inequalities for Dunkl Harmonic Oscillators}
\author{Huaiqian Li\footnote{Email: {\color{blue}huaiqianlee@gmail.com}}
\quad Bin Qian\footnote{Email: {\color{blue} binqiancn@126.com   }}
 \vspace{2mm}
\\
{\footnotesize Center for Applied Mathematics, Tianjin University, Tianjin 300072, P.R. China} \\
{\footnotesize Department of Mathematics and Statistics, Changshu Institute of Technology,} \\
{\footnotesize Changshu, Jiangsu 215500, P.R. China.}
}

\date{}

\usepackage{color} %%%
\def\R{\mathbb{R}}

\def\D{\mathbb{D}}

\def\d{\textup{d}}
\def\D{\textup{D}}

\def\e{\textup{e}}

\def\<{\langle}
\def\>{\rangle}
\def\Proof.{\noindent{\bf Proof. }}

\def\newdot{{\kern.8pt\cdot\kern.8pt}}

\newtheorem{theorem}{Theorem}[section]
\newtheorem{lemma}[theorem]{Lemma}
\newtheorem{corollary}[theorem]{Corollary}

\theoremstyle{definition}\newtheorem{remark}[theorem]{Remark}

\begin{document}
\allowdisplaybreaks
\maketitle
\makeatletter % '@' is now a normal "letter" for TeX
\renewcommand\theequation{\thesection.\arabic{equation}}
\@addtoreset{equation}{section}
\makeatother % '@' is restored as a "non-letter" character for TeX

\begin{abstract}
We study the Li--Yau inequality for the heat equation corresponding to the Dunkl harmonic oscillator, which is a non-local Schr\"{o}dinger operator parameterized by reflections and multiplicity functions. In the particular case when the reflection group is isomorphic to $\mathbb{Z}_2^d$, the result is sharp in the sense that equality is achieved by the heat kernel of the classic harmonic oscillator. We also provide the application on parabolic Harnack inequalities.
\end{abstract}



\section{Introduction and main results}\hskip\parindent
Let $M$ be a $d$-dimensional complete Riemannian manifold without boundary and with non-negative Ricci curvature, $\rho$ be the geodesic distance on $M$, $|\cdot|$ be the length in the tangent space, $\Delta$ be the Laplace--Beltrami operator and $\nabla$ be the Riemannian gradient. Let $V:[0,\infty)\times M\rightarrow\R$ such that $V(t,x)$ is $C^1$ in $t$ and $C^2$ in $x$. Suppose that there exist some point $o\in M$, a constant $\vartheta$ and functions $\eta:[0,\infty)\times[0,\infty)\rightarrow[0,\infty)$ and $\tau:[0,\infty)\rightarrow[0,\infty)$ such that
\begin{eqnarray*}
&&|\nabla V(t,\cdot)|(x)\leq \eta\big(\rho(o,x),t\big),\quad \Delta V(t,\cdot)(x)\leq\vartheta,\\
&&\lim_{r\rightarrow\infty}\frac{\eta(r,t)}{r}\leq \tau(t),
\end{eqnarray*}
for any $(t,x)\in(0,\infty)\times M$. In the seminal paper \cite{LiYau86}, Li and Yau obtained the following pointwise inequality, i.e., for every positive solution $u$ to the Schr\"{o}dinger equation
\begin{equation}\label{GHE}
\partial_t u =\Delta u - V u,\quad\mbox{on }(0,\infty)\times M,
\end{equation}
it holds that\footnote{Note that the term $\tau(t)^{2/3}$ appearing in the original result \cite[Theorem 1.3(i), page 163]{LiYau86} should be $\tau(t)^{1/2}$. We do not know if it has been pointed out somewhere else.}
\begin{eqnarray}\label{LY86-1}
\frac{|\nabla u(t,\cdot)|^2(x)}{u(t,x)^2}-\frac{\partial_t u(t,x)}{u(t,x)}\leq\frac{d}{2t}+\sqrt{\frac{d}{2}\vartheta}+ c\tau(t)^{\frac{1}{2}}+V(t,x),\quad t>0,\, x\in M,
\end{eqnarray}
for some constant $c>0$. Noting that $u$ is a positive solution to \eqref{GHE} and the fact that $\Delta \log u=\Delta u/u-|\nabla u|^2/u^2$, we immediately see that \eqref{LY86-1} is equivalent to
\begin{eqnarray}\label{LY86-2}
-\Delta\big(\log u(t,\cdot)\big)(x)\leq\frac{d}{2t} +\sqrt{\frac{d}{2}\vartheta}+ c\tau(t)^{1/2},\quad t>0,\, x\in M.
\end{eqnarray}
See e.g. \cite{Lijy91,Negrin95} for related studies in the Schr\"{o}dinger case. In particular, when $V$ vanishes, \eqref{LY86-1} and \eqref{LY86-2} reduce to the Li--Yau inequality for positive solutions to the heat equation $\partial_t u=\Delta u$ on $(0,\infty)\times M$, i.e.,
\begin{eqnarray}\label{LY86-3}
-\Delta\big(\log u(t,\cdot)\big)(x)=\frac{|\nabla u(t,\cdot)|^2(x)}{u(t,x)^2}-\frac{\partial_t u(t,x)}{u(t,x)}\leq\frac{d}{2t},\quad t>0,\, x\in M.
\end{eqnarray}

Being an important tool in the study of analytic and geometric properties of manifolds, Li--Yau inequalities have been successively applied to derive parabolic Harnack inequalities, estimate heat kernels and the Green functions, obtain eigenvalue estimates, and  establish Laplacian comparison theorems, etc.

It is well known that \eqref{LY86-3} is sharp in the sense that equality is achieved for the fundamental solution to the heat equation on the $d$-dimensional Euclidean space $\R^d$ (see e.g. \eqref{Gauss-kernel} below). Up to now, there are quite a few works on improving the Li--Yau inequality corresponding to heat equations for small time and large time on Riemannian manifolds; see e.g. \cite{Hamilton93,Davies89,fyw,Li-Xu11,BBG2017,YZ2020,ZhangQS} and references therein.

Recently, in \cite{YZ2020}, sharp Li--Yau inequalities for the Laplace--Beltrami operator on hyperbolic spaces were obtained by employing the explicit formula for the corresponding heat kernel. Very recently, in \cite{WZ2021}, similar to the idea of \cite{YZ2020}, the Li--Yau inequality for the fractional Laplacian has been proved; see also the conjectures on Li--Yau inequalities of gradient type for the fractional Laplacian at the end of \cite[Section 21]{Garo2018}, where the ``gradient'' is given by the carr\'{e} du champ operator induced by the fractional Laplacian. We should mention that there are works on the Li--Yau inequality in the setting of graphs via various curvature-dimension conditions in the sense of Bakry--Emery \cite{BE1985}; see e.g. \cite{BHLLMY,Qian2017,Munch2018,DKZ} and references there in.

So, a natural question is on the sharpness of the Li--Yau inequality \eqref{LY86-1} and \eqref{LY86-2} for the  Schr\"{o}dinger equation $\partial_t u=(\Delta-V)u$. Indeed, on $\R^d$, a typical example of the potential $V$ assumed above is $V(t,x)=|x|^2$. In the present work, we consider the Dunkl harmonic oscillator $L_\kappa$ on $\R^d$ given by the generalized Laplacian or Dunkl Laplacian $\Delta_\kappa$ and the potential $|x|^2$, i.e., $L_\kappa:=\Delta_\kappa-|x|^2$,
which is a non-local Schr\"{o}dinger operator; see Section 2 for more details. We mainly focus on the establishment of the sharp Li--Yau inequality for positive solutions to the equation $(\partial_t-L_\kappa)u=0.$

In the next section, we present some basics on the Dunkl theory and introduce our main results.

\section{Preparations and main results}\hskip\parindent
Let $|\cdot|$ and $\langle\cdot,\cdot\rangle$ be the norm and the scalar product on $\R^d$, respectively. For every $\alpha\in\R^d\setminus\{0\}$, denote the hyperplane orthogonal to $\alpha$ by $\alpha^\bot$, i.e., $\alpha^\bot=\{x\in\R^d: \langle\alpha,x\rangle=0\}$, and denote $\sigma_\alpha$ the reflection in $\alpha^\bot$ by
$$\sigma_\alpha x =x-2\frac{\langle \alpha,x\rangle}{|\alpha|^2}\alpha,\quad x\in\R^d.$$

Let $\mathcal{R}$ be a root system on $\R^d$, which is a finite set of nonzero vectors in $\R^d$ so that, for each $\alpha\in\mathcal{R}$, $\sigma_\alpha(\mathcal{R})=\mathcal{R}$ and $\alpha\R\cap \mathcal{R}=\{-\alpha,\alpha\}$, where $\alpha\R:=\{a\alpha: a\in\R\}$. Without loss of generality, we assume that $|\alpha|^2=2$ for every $\alpha\in\mathcal{R}$. Denote the  reflection group generated by the root system $\mathcal{R}$ by $G$. Let $\kappa: \mathcal{R}\rightarrow[0,\infty)$ be the multiplicity function such that $\kappa$ is $G$-invariant, i.e., $\kappa_{g\beta}=\kappa_\beta$ for every $g\in G$ and every $\beta\in\mathcal{R}$.

Fix a subsystem $\mathcal{R}_+$ of $\mathcal{R}$. For every $\xi\in\R^d$, the Dunkl operator $\D_\xi$ along $\xi$ associated with $G$ and $\kappa$, introduced by C.F. Dunkl in \cite{Dunkl1989},  is defined by
$$\D_\xi f(x)=\partial_\xi f(x)+\sum_{\alpha\in\mathcal{R}_+}\kappa_\alpha \langle\alpha,\xi\rangle \frac{f(x)-f(\sigma_\alpha x)}{\langle\alpha,x\rangle},\quad f\in C^1(\R^d),\,x\in\R^d,$$
where $\partial_\xi$ denotes the directional derivative along $\xi$. For convenience, write $\D_j$ for $\D_{\e_j}$ and $\partial_j$ for $\partial_{\e_j}$, $j=1,2,\cdots,d$, where $\{\e_j\}_{j=1}^d$ is the standard orthonormal system of $\R^d$. Let
$$\nabla_\kappa=(\D_1,\cdots,\D_d),\quad \Delta_\kappa=\sum_{j=1}^d \D_j^2$$
be the Dunkl gradient and the Dunkl Laplacian, respectively. It is easy to see that
$$\Delta_\kappa f(x)=\Delta f(x)+2\sum_{\alpha\in\mathcal{R}_+}\kappa_\alpha\left[\frac{\langle\alpha,\nabla f\rangle}{\langle\alpha,x\rangle} - \frac{f(x)-f(\sigma_\alpha x)}{\langle\alpha,x\rangle^2}\right],\quad f\in C^2(\R^d),\,x\in\R^d.$$
In particular, when $\kappa=0$, $\nabla_0=\nabla$ and $\Delta_0=\Delta$ are the standard gradient operator and the standard Laplacian on $\R^d$, respectively. However, $\Delta_\kappa$ and $\nabla_\kappa$ may not satisfy the chain rule and the Leibniz rule.

Let $\mu_{\kappa}(\d x)=w_\kappa(x)\d x$, where $\d x$ denotes the Lebesgue measure on $\R^d$ and $w_\kappa$ is the natural weight function defined by
$$w_\kappa(x)=\prod_{\alpha\in\mathcal{R}_+}|\langle \alpha,x \rangle|^{\kappa_\alpha},\quad x\in\R^d.$$
It is easy to see that $w_\kappa$ is $G$-invariant and a homogeneous function of degree
$$\lambda_\kappa=\sum_{\alpha\in\mathcal{R}_+}\kappa_\alpha.$$
The number $d+2\lambda_\kappa$ should be regarded as the homogeneous dimension of the Dunkl system due to the scaling property (see e.g. \cite[page 2365]{ADH2019}), i.e., for every ball $B(x,R)$ in $\R^d$ with center $x$ and radius $R>0$, it is easy to see that
$$\mu_\kappa\big(B(\gamma x,\gamma R)\big)=\gamma ^{d+2\lambda_\kappa}\mu_\kappa\big(B(x,R)\big),\quad \gamma>0.$$

For more details on the Dunkl theory, refer to the survey papers \cite{Rosler2003,Anker2017} and the books \cite{DunklXu2014,DX2015}. For the application of the Dunkl theory in mathematical physics, see \cite[Section 3]{Rosler2003} and references therein. Moreover, from the probabilistic point of view, the stochastic process (also called Dunkl process) corresponding to the Dunkl Laplacian $\Delta_\kappa$ is a Markov jump process but not a L\'{e}vy process if $\kappa>0$. See e.g. \cite{RosVoi1998}, \cite{GaYor2005} and \cite[Section 3]{LZ2020} for some probabilistic aspects of the Dunkl theory  and see e.g. \cite{Sato1999} for more details on L\'{e}vy processes.

In this work, we mainly consider the Dunkl harmonic oscillator on $\R^d$, i.e.,
$$L_\kappa:=\Delta_\kappa-|x|^2.$$
which clearly reduces to the classic harmonic oscillator $L:=\Delta-|x|^2$ when $\kappa=0$. Let $(H_t)_{t\geq0}$ be the semigroup generated by $L_\kappa$, and let  $(h_t)_{t>0}$ be the corresponding heat kernel with respect to $\mu_\kappa$. See e.g. \cite{NS2009} and Section 3 for more details.

Now we are ready to present our main results. The first one is on the Li--Yau inequality for the heat kernel of the Dunkl harmonic oscillator in the particular case when the reflection group $G$ is isomorphic to the Abelian group $\mathbb{Z}_2^d$.
\begin{theorem}\label{main-thm-1}
Suppose that $G$ is isomorphic to $\mathbb{Z}_2^d$. Then, for every $t>0$ and every $x,y\in\R^d$,
\begin{equation}\begin{split}\label{LY-dhc-1}
-\Delta_\kappa\big(\log h_t(\cdot,y)\big)(x)&\leq (d+2\lambda_\kappa)\coth(2t)\\
&\leq (d+2\lambda_\kappa)\left(\frac{1}{2t}+1\right).
\end{split}\end{equation}
\end{theorem}
\begin{remark}
The result improves \eqref{LY86-2} and is sharp in the sense that equality in the first inequality of \eqref{LY-dhc-1} is achieved by the heat kernel of the harmonic oscillator. Indeed, if $\kappa=0$, then $L_\kappa$ becomes the harmonic oscillator $L$. It is well known that the heat kernel $h_t(x,y)$ corresponding to $L$ has the following expression, i.e.,
$$h_t(x,y)=\frac{1}{[2\pi\sinh(2t)]^d}\exp\left(-\frac{1}{4}\left[\tanh(t)|x+y|^2+\coth(t)|x-y|^2\right]\right),$$
for every $t>0,\,x,y\in\R^d$; see e.g. \cite[page 453]{StemTor2003}. Then it is easy to see that
\begin{eqnarray*}
-\Delta\big(\log h_t(\cdot,y)\big)(x)&=&d \coth(2t)\\
&\leq&d\left(1+\frac{1}{2t}\right),\quad t>0,\,x,y\in\R^d.
\end{eqnarray*}
\end{remark}

In order to obtain the sharp Li--Yau inequality for solutions to the equation $\partial_t u=L_\kappa u$, we rely on a general result which %reduces the problem to
transforms the problem into an equivalent one, i.e., establishing the Li--Yau inequality for the corresponding heat kernel. The idea is motivated by the continuous setting on Riemannian manifolds in \cite[Theorem 1.1]{YZ2020} and the pure jump setting on metric measure spaces in \cite[Theorem 2.4]{WZ2021}.

From now on, we fix $T\in(0,\infty]$. Let $V:(0,T)\times\R^d\rightarrow\R$ be a potential. Consider the equation associated with the non-local Schr\"{o}dinger operator $L_V:=\Delta_\kappa-V$, i.e.,
\begin{equation}\label{Dun-Sch-equ}
\partial_t u(t,x)=L_Vu(t,\cdot)(x),\quad (t,x)\in(0,T)\times\R^d.
\end{equation}
Let $\mathcal{S}(L_V)$ be the class of all solutions $u:[0,T)\times\R^d\rightarrow(0,\infty)$ to \eqref{Dun-Sch-equ} such that
$(0,T)\ni t\mapsto u(t,x)$ is $C^1$ for every $x\in \R^d$ and $\R^d\ni x\mapsto u(t,x)$ is $C^2$ for every $t\in(0,T)$.

In the next main result, we will make the following hypothesis.
\begin{itemize}
\item[({\rm \textbf{H}})] For each $u\in\mathcal{S}(L_V)$, there exists a function $h^V: (0,T)\times\R^d\times\R^d\rightarrow(0,\infty)$ such that, $y\mapsto h_t^V(x,y)$ is Borel measurable for every $(t,x)\in(0,T)\times\R^d$, $(t,x)\mapsto h_t^V(x,y)\in\mathcal{S}(L_V)$ for every $y\in\R^d$, and
 \begin{equation*}\begin{split}
 u(t,x)&=\int_{\R^d}u(0,y)h^V_t(x,y)\,\d\mu_\kappa(y),\\
 \partial_t u(t,x)&=\int_{\R^d}u(0,y)\partial_t h^V_t(x,y)\,\d\mu_\kappa(y),
\end{split}\end{equation*}
for every $(t,x)\in(0,T)\times\R^d$.
\end{itemize}

%The second main result is on the Li--Yau inequality for solutions of the equation associated with the Dunkl Schr\"{o}dinger operator $L_V:=\Delta_\kappa-V$, which in particular implies the Li--Yau inequality holds for solutions of the equation associated with the Dunkl harmonic oscillator.
\begin{theorem}\label{main-thm-2}
Let $\beta$ be a function defined on $(0,T)\times \R^d$. Assume that $\mathcal{S}(L_V)\neq\emptyset$ and $({\rm \textbf{H}})$ holds. Then
\begin{eqnarray}\label{thm-2-0}
-\Delta_\kappa\big(\log h_t^V(\cdot,y)\big)(x)\leq \beta(t,x), \quad (t,x,y)\in (0,T)\times \R^d\times \R^d,
\end{eqnarray}
is equivalent to that, for every $u\in\mathcal{S}(L_V)$,
\begin{eqnarray}\label{thm-2-1}
-\Delta_\kappa\big(\log u(t,\cdot)\big)(x)\leq \beta(t,x), \quad (t,x)\in (0,T)\times \R^d.
\end{eqnarray}
In addition, either \eqref{thm-2-0} or \eqref{thm-2-1} implies that, for every $u\in\mathcal{S}(L_V)$,
\begin{eqnarray}\label{thm-2-1+}
\frac{|\nabla u(t,\cdot)|^2(x)}{u(t,x)^2}-\frac{\partial_t u(t,x)}{u(t,x)}\leq  \beta(t,x)+ V(t,x), \quad (t,x)\in (0,T)\times \R^d.
\end{eqnarray}
\end{theorem}
%{\color{red}\begin{remark}
%It is possible to write a more general form of Theorem \ref{main-thm-2} as in \cite{WZ2021}.
%\end{remark}}

Taking $V(t,x)=|x|^2$ for all $t\in[0,T)$, by Theorems \ref{main-thm-2} and \ref{main-thm-1}, we immediately obtain the following Li--Yau inequality. We say that a function $f:[0,T)\times\R^d\rightarrow I$ is $C^{1,2}$ if $(0,T)\ni t\mapsto f(t,x)$ is $C^1$ for every $x\in \R^d$ and $\R^d\ni x\mapsto u(t,x)$ is $C^2$ for every $t\in(0,T)$, where $I$ is a subset of $\R$.
\begin{corollary}\label{LY-dho}
Suppose that $G$ is isomorphic to $\mathbb{Z}_2^d$. Then for every $C^{1,2}$ solution $u:[0,T)\times \R^d\rightarrow(0,\infty)$ to the equation $\partial_tu=L_\kappa u$, it holds that
\begin{equation}\begin{split}\label{LY-dho-1}
-\Delta_\kappa\big(\log u(t,\cdot)\big)(x)&\leq (d+2\lambda_\kappa)\coth(2t)\\
&\leq (d+2\lambda_\kappa)\left(\frac{1}{2t}+1\right),\quad T>t>0,\,x\in\R^d.
\end{split}\end{equation}
\end{corollary}
\begin{remark}
Consider the case when $\kappa=0$. Then $L_\kappa$ reduces to $L=\Delta-|x|^2$ and $\lambda_\kappa=0$. Note that even the weaker inequality in \eqref{LY-dho-1} improves \eqref{LY86-2} by getting rid of a positive constant on the right hand side of  \eqref{LY86-2}.
\end{remark}

An important direct consequence of the above Li--Yau inequalities is the following sharp parabolic Harnack inequalities.
\begin{corollary}\label{harnack}
Assume that $u:[0,T)\times \R^d\rightarrow (0,\infty)$ is a $C^{1,2}$ solution to the equation  $\partial_tu=L_\kappa u$ and the reflection group $G$ is isomorphic to $\mathbb{Z}_2^d$. Then for every $0<s<t<T$ and every $x,y\in\R^d$,
\begin{eqnarray*}
u(s,x)&\leq& u(t,y)\left(\frac{\sinh(2t)}{\sinh(2s)}\right)^{\frac{d+2\lambda_\kappa}{2}}
\exp\left(\frac{|x-y|^2}{4(t-s)}+(t-s)\varsigma(x,y)\right)\\
&\leq& u(t,y)\left(\frac{t}{s}\right)^{\frac{d+2\lambda_\kappa}{2}}
\exp\left(\frac{|x-y|^2}{4(t-s)}+(t-s)\big[d+2\lambda_\kappa+\varsigma(x,y)\big]\right),
\end{eqnarray*}
where
$$\varsigma(x,y):=\frac{|x|^2+|y|^2+\langle x,y\rangle}{3},\quad x,y\in\R^d.$$
\end{corollary}

We give a remark on Li--Yau inequalities and parabolic Harnack inequalities above. The proof will be postponed to Section 4.
\begin{remark}\label{rmk}
Suppose that the assumption of Corollary \ref{harnack} holds.
\begin{itemize}
\item[(1)] Let $\varrho_t(x)=h_t(x,0)$. Then  for every $t\in (0,T)$, $\R^d\ni x\mapsto \log(\frac{u_t}{\varrho_t})(x)$ is convex. Refer to \cite[Corollary 1.4]{HeTop2013} for the case of standard heat equation in $\R^d$.

\item[(2)] Let $\beta: (0,T)\times \R^d\rightarrow\R$ be jointly continuous. Then the Li--Yau inequality
\begin{equation}\label{Rem-LY}
\frac{|\nabla u(t,\cdot)|^2(x)}{u(t,x)^2}-\frac{\partial_t u(t,x)}{u(t,x)}\leq  \beta(t,x),\quad (t,x)\in (0,T)\times \R^d
\end{equation}
is equivalent to the parabolic Harnack inequality
\begin{equation}\label{Rem-Har}
u(s,x)\le u(t,y)\exp\left\{\frac{|x-y|^2}{4(t-s)}+(t-s)\int_0^1\beta(t+\tau(s-t),y+\tau(x-y))d\tau\right\},
\end{equation}
for any   $x,y\in \R^d,\, 0<s<t<T$.
\end{itemize}
\end{remark}

Now we turn to consider the case of heat equation associated with the Dunkl Laplacian $\Delta_\kappa$. %Let $(P_t)_{t\geq0}$ be the semigroup generated by $\Delta_\kappa$, and let
Let $(p_t)_{t>0}$ be the heat kernel corresponding to $\Delta_\kappa$ with respect to $\mu_\kappa$, the so called Dunkl heat kernel  (see Section 3 for more details).

The next result is on the Li--Yau inequality for $(p_t)_{t>0}$, which can be proved by the same method as for Theorem \ref{main-thm-1}.
\begin{theorem}\label{dunkl-main-thm-1}
Assume that $G$ is isomorphic to $\mathbb{Z}_2^d$. Then for every $t>0$ and every $x,y\in\R^d$,
\begin{eqnarray}\label{LY-1}
-\Delta_\kappa\big(\log p_t(\cdot,y)\big)(x)\leq \frac{d+2\lambda_\kappa}{2t}.
\end{eqnarray}
\end{theorem}
\begin{remark}
\eqref{LY-1} is sharp in the sense that, if $\kappa=0$, then $\Delta_\kappa=\Delta$, $\lambda_\kappa=0$, and $(p_t)_{t>0}$ reduce to the heat kernel corresponding to $\Delta$ on $\R^d$, i.e.,
\begin{eqnarray}\label{Gauss-kernel}
p_t(x,y)=\frac{1}{(4\pi t)^{d/2}}\exp\Big(-\frac{|x-y|^2}{4t}\Big),\quad x,y\in\R^d,\,t>0;
\end{eqnarray}
hence
$$-\Delta\big(\log p_t(\cdot,y)\big)(x)=\frac{d}{2t},$$
for every $t>0$ and every $x,y\in\R^d$.
\end{remark}

Combining Theorem \ref{main-thm-2} with vanishing potential $V$ and Theorem \ref{dunkl-main-thm-1} together, we have the following Li--Yau inequality for positive solutions to the Dunkl heat equation.  Then the corresponding parabolic Harnack inequality follows immediately  by the same approach as for Corollary \ref{harnack}.
\begin{corollary}\label{dunkli-LY-dho}
Suppose that $G$ is isomorphic to $\mathbb{Z}_2^d$. Then for every $C^{1,2}$ solution $u:[0,T)\times \R^d\rightarrow(0,\infty)$ to the Dunkl heat equation $\partial_tu=\Delta_\kappa u$, it holds that
\begin{equation*}
-\Delta_\kappa\big(\log u(t,\cdot)\big)(x)\leq \frac{d+2\lambda_\kappa}{2t},\quad (t,x)\in (0,T)\times\R^d,
\end{equation*}
and moreover,  for every $0<s<t<T$ and every $x,y\in\R^d$,
\begin{eqnarray*}
u(s,x)\leq u(t,y)\left(\frac{t}{s}\right)^{\lambda_\kappa+d/2}\exp\left\{\frac{|x-y|^2}{4(t-s)}\right\}.
\end{eqnarray*}
\end{corollary}

The proofs of Theorems \ref{main-thm-1} and \ref{dunkl-main-thm-1} are presented in Section 3, and proofs of Theorem \ref{main-thm-2}, Corollary \ref{harnack} and Remark \ref{rmk} are given in Section 4.

\section{Proofs of Theorems \ref{main-thm-1} and \ref{dunkl-main-thm-1}}\hskip\parindent
In this section, we assume that the reflection group $G$ is isomorphic to $\mathbb{Z}_2^d$. %We should mention that there are many studies in this setting; see e.g. \cite{NS2009,DX2015,ABDH2015}.
In this case, the root system $\mathcal{R}=\{\pm\sqrt{2}\e_j: j=1,\cdots,d\}$ and the subsystem $\mathcal{R}_+=\{\sqrt{2}\e_j: j=1,\cdots,d\}$. The reflection group $G$ is generated by $\{\sigma_j: j=1,\cdots,d\}$, where
$$\sigma_i(x_1,\cdots,x_i,\cdots,x_d)=(x_1,\cdots,x_{i-1},-x_i,x_{i+1},\cdots,x_d),\quad i=1,\cdots,d,$$
for each $x\in\R^d$ with $x=(x_1,\cdots,x_d)\in\R^d$. Since the multiplicity function $\kappa$ is nonnegative and $\mathbb{Z}_2^d$-invariant, for each $j=1,\cdots,d$, we may take $\kappa_{\sigma_j}=\kappa_j\geq0$. Then, for every $x=(x_1,\cdots,x_d)\in\R^d$,
$$\D_j f(x)=\partial_j f(x)+\frac{\kappa_j}{x_j}\big[f(x)-f(\sigma_j x)\big],\quad f\in C^1(\R^d),\,j=1,\cdots,d,$$
and
\begin{eqnarray*}
\Delta_\kappa f(x)&=&\Delta f(x)+\sum_{j=1}^d \frac{\kappa_j}{x_j^2}\big[2x_j\partial_j f(x)-f(x)+f(\sigma_j x)\big]\cr
&=&\sum_{j=1}^d\left(\partial_{jj}^2 f(x) + \frac{\kappa_j}{x_j^2}\big[2x_j\partial_j f(x)-f(x)+f(\sigma_j x)\big]\right),\quad f\in C^2(\R^d),
\end{eqnarray*}
The weight function can be written as
$$w_\kappa(x)=\prod_{j=1}^d|x_j|^{\kappa_j},\quad x\in\R^d,$$
which is clearly homogeneous of degree $\lambda_\kappa=\sum_{j=1}^d \kappa_j$.

It has been shown in \cite[page 544]{NS2009} (see also \cite{Rosler1998}) that the heat kernel $h_t(x,y)$ corresponding to the Dunkl harmonic oscillator $L_\kappa=\Delta_\kappa-|x|^2$ with respect to $\mu_\kappa$ has the following expression, i.e., for every $t>0$ and every $x,y\in\R^d$ with $x=(x_1,\cdots,x_d),y=(y_1,\cdots,y_d)$,
\begin{equation}\begin{split}\label{hc-kernel-1}
h_t(x,y)&=\frac{1}{[2\sinh(2t)]^d}\exp\left(-\frac{\coth(2t)}{2}\left[|x|^2+|y|^2\right]\right)\\
&\quad\times\prod_{j=1}^d
\left[\frac{I_{\kappa_j-1/2}(\frac{x_jy_j}{\sinh(2t)})}{(x_jy_j)^{\kappa_j-1/2}}
+x_jy_j\frac{I_{\kappa_j+1/2}(\frac{x_jy_j}{\sinh(2t)})}{(x_jy_j)^{\kappa_j+1/2}}\right],
\end{split}\end{equation}
where for every $\lambda>-1$, $I_\lambda$ is the modified Bessel function of the first kind with order $\lambda$ defined as
$$z^{-\lambda}I_\lambda(z)=\sum_{j=0}^\infty \frac{1}{j!\Gamma(\lambda+j+1)}\Big(\frac{z}{2}\Big)^{2j},\quad z\in\R,$$
and $\Gamma(\cdot)$ stands for the Gamma function. Note that $z\mapsto z^{-\lambda}I_\lambda(\cdot)$ is smooth on $\R$ and $I_\lambda(z)>0$ for every  $z\in(0,\infty)$; see e.g. \cite[Chapter 5]{Lebedev72}.

Let $\delta_\cdot$ be the Dirac measure. Due to Schl\"{a}fli's integral representation of the modified Bessel function $I_\lambda$ (see e.g. \cite[page 549]{NS2009}), we see that
$$I_\lambda(z)=z^\lambda\int_{-1}^1e^{-zs}\,\nu_\lambda(\d s),\quad z>0,\,\lambda\geq -\frac{1}{2},$$
where
$$\nu_\lambda(\d s):=\frac{(1-s^2)^{\lambda-1/2}}{\sqrt{\pi}2^\lambda\Gamma(\lambda+1/2)}\,\d s,\quad s\in(-1,1),\,\lambda>-\frac{1}{2},$$
and
$$\nu_{-1/2}(\d s):=\frac{1}{\sqrt{2\pi}}(\delta_{-1}+\delta_{1}),$$
Then, by \cite[Example 2.1, page 107]{Rosler1998}, we observe that \eqref{hc-kernel-1} can be rewritten  as
\begin{equation}\begin{split}\label{hc-kernel-2}
h_t(x,y)&=\prod_{j=1}^d \frac{c_{\kappa_j}}{\left[\sinh(2t)\right]^{\kappa_j+1/2}}\exp\left(-\frac{1}{2}\coth(2t)(x_j^2+y_j^2)\right)\\
&\quad\times\int_{-1}^1(1-s)^{\kappa_j-1}(1+s)^{\kappa_j}\exp\Big(\frac{x_jy_js}{\sinh(2t)}\Big)\,\d s\\
&=:\prod_{j=1}^d h_t^{(j)}(x_j,y_j),
\end{split}\end{equation}
for every $t>0$ and every $x,y\in\R^d$ with $x=(x_1,\cdots,x_d),y=(y_1,\cdots,y_d)$, where $c_{\kappa_j}^{-1}=2^{\kappa_j+1/2} \sqrt{\pi} \Gamma(\kappa_j+1/2)$.

Now we are ready to present the proof.
\begin{proof}[Proof of Theorem \ref{main-thm-1}] Let $t>0$ and $x,y\in\R^d$ with $x=(x_1,\cdots,x_d), \, y=(y_1,\cdots,y_d)$. By \eqref{hc-kernel-2}, it is clear that
\begin{equation}\begin{split}\label{thm-1-proof-1}
&\Delta_\kappa\big(\log h_t(\cdot,y)\big)(x)\\
&=\sum_{j=1}^d \Big(\frac{\kappa_j}{x_j^2}\big[2x_j\partial_{x_j}\log h_t^{(j)}(x_j,y_j)
-\log h_t^{(j)}(x_j,y_j)+ \log h_t^{(j)}(-x_j,y_j)\big]\\
&\quad+\partial_{x_jx_j}^2 \log h_t^{(j)}(x_j,y_j)\Big).
\end{split}\end{equation}
Hence, it suffices to estimate the terms in the parentheses of \eqref{thm-1-proof-1}, i.e.,
$$\partial_{x_jx_j}^2 \log h_t^{(j)}(x_j,y_j)+\frac{\kappa_j}{x_j^2}\big[2x_j\partial_{x_j}\log h_t^{(j)}(x_j,y_j)
-\log h_t^{(j)}(x_j,y_j)+ \log h_t^{(j)}(-x_j,y_j)\big].$$
For notational convenience, we ignore the superscript and the subscript $j$ in the above terms. Then we shall estimate
\begin{eqnarray*}
\partial_{uu}^2 \log h_t(u,v)&+&\frac{\kappa}{u^2}\big[2u\partial_{u}\log h_t(u,v)-\log h_t(u,v)+ \log h_t(-u,v)\big],
\end{eqnarray*}
for every $t>0$ and $u,v\in\R$, where $\kappa$ is a nonnegative constant,
\begin{equation*}\begin{split}\label{hc-kernel-1D}
h_t(u,v)&=\frac{c_{\kappa}}{\left[\sinh(2t)\right]^{\kappa+1/2}}\exp\left(-\frac{1}{2}\coth(2t)(u^2+v^2)\right)\\
&\quad\times\int_{-1}^1(1-s)^{\kappa-1}(1+s)^{\kappa}\exp\Big(\frac{uvs}{\sinh(2t)}\Big)\,\d s,
\end{split}\end{equation*}
  and $c_{\kappa}=\frac{1}{2^{\kappa+1/2}\Gamma(\kappa+1/2)\sqrt{\pi}}$.

Let $g(s)=(1-s)^{\kappa-1}(1+s)^{\kappa}$ and
$$E(u,v)=\int_{-1}^1g(s)\exp\Big(\frac{uvs}{\sinh(2t)}\Big)\,\d s.$$
Then
\begin{equation}\begin{split}\label{log-kernel}
\log h_t(u,v)=\log c_\kappa-\left(\kappa+\frac{1}{2}\right)\log\sinh(2t)-\frac{1}{2}\coth(2t)\left(u^2+v^2\right)+\log E(u,v),
\end{split}\end{equation}
for every $t>0$ and every $u,v\in\R$.

\textbf{Step 1}. Estimate $\partial_{uu}^2 \log h_t(u,v)$ from below. From \eqref{log-kernel}, we immediately have
\begin{equation*}\begin{split}
\partial_{uu}^2 \log h_t(u,v)=-\coth(2t)+\frac{\partial_{uu}^2 E(u,v)}{E(u,v)}-\frac{\big(\partial_{u}E(u,v)\big)^2}{E(u,v)^2}
\end{split}\end{equation*}
Since
\begin{eqnarray*}
\partial_u E(u,v)&=&\frac{v}{\sinh(2t)}\int_{-1}^1s g(s)\exp\Big(\frac{uvs}{\sinh(2t)}\Big)\,\d s,\\
\partial_{uu}^2 E(u,v)&=&\frac{v^2}{[\sinh(2t)]^2}\int_{-1}^1s^2g(s)\exp\Big(\frac{uvs}{\sinh(2t)}\Big)\,\d s,
\end{eqnarray*}
we obtain
\begin{equation*}\begin{split}
&\partial_{uu}^2 E(u,v)E(u,v)-\big(\partial_{u}E(u,v)\big)^2\\
&=\frac{v^2}{[\sinh(2t)]^2}\left(\int_{-1}^1s^2g(s)\exp\Big(\frac{uvs}{\sinh(2t)}\Big)\,\d s\right)\left(\int_{-1}^1g(s)\exp\Big(\frac{uvs}{\sinh(2t)}\Big)\,\d s\right)\\
&\quad-\left(\frac{v}{[\sinh(2t)]}\int_{-1}^1sg(s)\exp\Big(\frac{uvs}{\sinh(2t)}\Big)\,\d s\right)^2\\
&\geq0,
\end{split}\end{equation*}
by the Cauchy--Schwarz inequality. Thus
\begin{eqnarray}\label{proof-thm-1-1}
\partial_{uu}^2\log h_t(u,v)\geq-\coth(2t).
\end{eqnarray}

\textbf{Step 2}. Estimate ${\rm I}$ from below, where
$${\rm I}:=\frac{\kappa}{u^2}\big[2u\partial_u\log h_t(u,v)+\log h_t(-u,v)-\log h_t(u,v)\big].$$
Then, letting $a=uv/\sinh(2t)$, we have
\begin{eqnarray}\label{eq-J}
{\rm I}&=&\frac{\kappa}{u^2}\left[-2\coth(2t)u^2+ 2u \frac{\partial_u E(u,v)}{  E(u,v)} + \log\frac{ E(-u,v)}{  E(u,v)}\right]\cr
&=&\frac{\kappa}{u^2}\left[-2\coth(2t)u^2+ 2a \frac{\int_{-1}^1 sg(s) e^{as}\,\d s}{\int_{-1}^1 g(s) e^{as}\,\d s  } + \log\frac{\int_{-1}^1 g(s) e^{-as}\,\d s}{  \int_{-1}^1 g(s) e^{as}\,\d s}\right].
\end{eqnarray}

Set
$$\phi(a)=2a \frac{\int_{-1}^1 sg(s) e^{as}\,\d s}{\int_{-1}^1 g(s) e^{as}\,\d s  } + \log\frac{\int_{-1}^1 g(s) e^{-as}\,\d s}{  \int_{-1}^1 g(s) e^{as}\,\d s},\quad a\in\R.$$
We \textbf{claim} that $$\phi(a)\geq 0, \quad a\in\R.$$

Now we begin to prove the claim. Indeed, $\phi(0)=0$, and
\begin{eqnarray*}
\phi'(a)&=&2a \left[ \frac{\int_{-1}^1 s^2g(s) e^{as}\,\d s}{\int_{-1}^1 g(s) e^{as}\,\d s} -\frac{\big(\int_{-1}^1 sg(s) e^{as}\,\d s\big)^2}{\big(\int_{-1}^1 g(s) e^{as}\,\d s\big)^2} \right]\\
&&+\frac{\int_{-1}^1 sg(s) e^{as}\,\d s}{\int_{-1}^1 g(s) e^{as}\,\d s}-\frac{\int_{-1}^1 sg(s) e^{-as}\,\d s}{\int_{-1}^1 g(s) e^{-as}\,\d s}.
\end{eqnarray*}
Applying the Cauchy--Schwarz inequality, we have
\begin{eqnarray}\label{proof-J}
\frac{\int_{-1}^1 s^2g(s) e^{as}\,\d s}{\int_{-1}^1 g(s) e^{as}\,\d s} -\frac{\big(\int_{-1}^1 sg(s) e^{as}\,\d s\big)^2}{\big(\int_{-1}^1 g(s) e^{as}\,\d s\big)^2}\geq0.
\end{eqnarray}
For every $a\in\R$, set
\begin{equation*}\begin{split}
\psi(a)&:=\Big(\int_{-1}^1 sg(s) e^{as}\,\d s\Big)\Big(\int_{-1}^1 g(s) e^{-as}\,\d s\Big) \\
 &\quad-  \Big(\int_{-1}^1 g(s) e^{as}\,\d s\Big)
\Big(\int_{-1}^1 sg(s) e^{-as}\,\d s\Big).
\end{split}\end{equation*}
Then $\psi(0)=0$, and
\begin{eqnarray*}
\psi'(a)&=&\Big(\int_{-1}^1 s^2g(s) e^{as}\,\d s\Big)\Big(\int_{-1}^1 g(s) e^{-as}\,\d s\Big)\\
 &&-2\Big(\int_{-1}^1 sg(s) e^{as}\,\d s\Big)\Big(\int_{-1}^1 sg(s) e^{-as}\,\d s\Big)\\
&&+\Big(\int_{-1}^1 g(s) e^{as}\,\d s\Big)\Big(\int_{-1}^1 s^2g(s) e^{-as}\,\d s\Big)\\
&\geq&AD-2\sqrt{ABCD}+BC=\big(\sqrt{AD}-\sqrt{BC} \big)^2\\
&\geq&0,
\end{eqnarray*}
where we applied the Cauchy--Schwarz inequality twice and set
\begin{eqnarray*}
A&=&\int_{-1}^1 g(s) e^{as}\,\d s,\quad\,\,\,\,\,\, B=\int_{-1}^1 g(s) e^{-as}\,\d s,\\
C&=&\int_{-1}^1 s^2g(s) e^{as}\,\d s,\quad D=\int_{-1}^1 s^2g(s) e^{-as}\,\d s.
\end{eqnarray*}
Hence $a\mapsto \psi(a)$ is increasing in $\R$.

(1) Suppose $a\geq0$. Then, by \eqref{proof-J},
\begin{eqnarray*}
\phi'(a)&\geq&\frac{\int_{-1}^1 sg(s) e^{as}\,\d s}{\int_{-1}^1 g(s) e^{as}\,\d s} -  \frac{\int_{-1}^1 sg(s) e^{-as}\,\d s}{\int_{-1}^1 g(s) e^{-as}\,\d s}\\
&=&\frac{\psi(a)}{\big(\int_{-1}^1 g(s) e^{as}\,\d s\big)\big(\int_{-1}^1 g(s) e^{-as}\,\d s\big)}.
\end{eqnarray*}
Since $a\mapsto \psi(a)$ is increasing in $[0,\infty)$, we have $\psi(a)\geq \psi(0)=0$, $a\geq0$. Hence $\phi'(a)\geq0$, $a\geq0$.

(2) Suppose $a\leq0$. Then, by \eqref{proof-J},
\begin{eqnarray*}
\phi'(a)&\leq&\frac{\int_{-1}^1 sg(s) e^{as}\,\d s}{\int_{-1}^1 g(s) e^{as}\,\d s} -  \frac{\int_{-1}^1 sg(s) e^{-as}\,\d s}{\int_{-1}^1 g(s) e^{-as}\,\d s}\\
&=&\frac{\psi(a)}{\big(\int_{-1}^1 g(s) e^{as}\,\d s\big)\big(\int_{-1}^1 g(s) e^{-as}\,\d s\big)}.
\end{eqnarray*}
Since $a\mapsto \psi(a)$ is increasing in $(-\infty,0]$, we have $\psi(a)\leq \psi(0)=0$, $a\leq0$. Hence $\psi'(a)\leq0$, $a\leq0$.

Combining (1) and (2), we see that $\phi'(a)\geq0$ for every $a\geq0$, and $\phi'(a)\leq0$ for every $a\leq0$. Thus, $\phi(a)\geq \phi(0)=0$ for every $a\in\R$, which completes the proof of the claim.

Thus, by the claim and \eqref{eq-J}, we have
\begin{eqnarray}\label{proof-thm-1-2}
{\rm I}\geq -2\kappa\coth(2t).
\end{eqnarray}

Therefore, gathering \eqref{thm-1-proof-1}, \eqref{proof-thm-1-1} and \eqref{proof-thm-1-2} together, we obtain
\begin{eqnarray*}
-\Delta_\kappa\big(\log h_t(\cdot,y)\big)(x)&\leq& (d+2\lambda_\kappa)\coth(2t)\\
&\leq& (d+2\lambda_\kappa)\left(1+\frac{1}{2t}\right),\quad t>0,\,x,y\in\R^d,
\end{eqnarray*}
where in the last inequality we used the elementary fact that $e^s\geq 1+s$ for every $s\geq0$.

We complete the proof of \eqref{LY-1}.
\end{proof}

In order to apply the above method to prove Theorem \ref{dunkl-main-thm-1}, we only need to observe the following fact (see e.g. \cite{Dunkl1992,Rosler1998}).

For every $t>0$ and every $x,y\in\R^d$ with $x=(x_1,\cdots,x_d)$ and $y=(y_1,\cdots,y_d)$,
\begin{eqnarray*}\label{dunkl-product-kernel}
p_t(x,y)=\prod_{i=1}^dp^{(i)}_t(x_i,y_i),
\end{eqnarray*}
where for each $i=1,\cdots,d$, $\kappa_i$ is a nonnegative constant, and
\begin{eqnarray*}\label{dunkl-1d-kernel}
p^{(i)}_t(u,v)&:=&\frac{1}{\Gamma(\kappa_i+1/2)(2t)^{\kappa_i+1/2}}\exp\left(-\frac{u^2+v^2}{4t}\right)
E_{\kappa_i}\Big(\frac{u}{\sqrt{2t}},\frac{v}{\sqrt{2t}}\Big),\\
E_{\kappa_i}(u,v)&:=&\frac{\Gamma(\kappa_i+1/2)}{\sqrt{\pi} \Gamma(\kappa_i)}\int_{-1}^1 (1-s)^{\kappa_i-1}(1+s)^{\kappa_i} e^{suv}\,\d s,
\end{eqnarray*}
for every $u,v\in\R$ and $t>0$. In the present $\mathbb{Z}_2^d$ setting, also refer to \cite{ABDH2015} for more details on the Dunkl heat kernel and its estimates and see \cite{DX2015} for the development of harmonic analysis in the Dunkl setting.

\section{Proofs of Theorem \ref{main-thm-2}, Corollary \ref{harnack} and Remark \ref{rmk}}\hskip\parindent
For $\psi\in C^1(\R)$ and $a,b\in\R$, let
$$\pi_\psi(a,b):=\psi(a)-\psi(b)-\psi'(b)(a-b).$$
In order to prove Theorem \ref{main-thm-2}, we need the following lemma which is motivated by \cite[Lemma 4.4]{GLR2018} in the particular case when $\psi(t)=|t|^p$ with $p>1$ for any $t\in\R$, and see also the recent \cite[Lemma 2.1]{WZ2021} for the general pure jump case.
\begin{lemma}\label{chain-rule}
Let $I\subseteq\R$ be an interval, $\psi\in C^2(I)$ and $f\in C^2(\R^d,I)$. Then,
\begin{eqnarray*}
\Delta_\kappa\psi(f)=\psi'(f)\Delta_\kappa f+\psi''(f)|\nabla f|^2+\Pi_\psi(f),
\end{eqnarray*}
where
$$\Pi_\psi(f)(x):=2\sum_{\alpha\in\mathcal{R}_+}\kappa_\alpha\frac{\pi_\psi\big(f(\sigma_\alpha x),f(x)\big)}{\langle\alpha,x\rangle^2},\quad x\in\R^d.$$
In addition, if $f$ is $G$-invariant, i.e., $f(g x)=f(x)$ for every $g\in G$ and every $x\in\R^d$, then
$$\Delta_\kappa\psi(f)=\psi'(f)\Delta f+\psi''(f)|\nabla f|^2.$$
\end{lemma}
\begin{proof}
We only need to prove the first assertion, which is derived by a direct calculation. For every $x\in\R^d$,
\begin{eqnarray*}
\Delta_\kappa\psi(f)(x)&=&\Delta\psi(f)(x)+2\sum_{\alpha\in\mathcal{R}_+}\kappa_\alpha
\bigg(\frac{\langle\alpha,\nabla\psi\big(f(x)\big)\rangle}{\langle\alpha,x\rangle}+\frac{\psi\big(f(\sigma_\alpha x)\big)-\psi\big(f(x)\big)}{\langle\alpha,x\rangle^2}\bigg)\\
&=&\psi'\big(f(x)\big)\Delta f(x)+\psi''\big(f(x)\big)|\nabla f|^2(x)\\
&&+2\sum_{\alpha\in\mathcal{R}_+}\kappa_\alpha
\bigg(\psi'\big(f(x)\big)\frac{\langle\alpha,\nabla f(x) \rangle}{\langle\alpha,x\rangle}
+\psi'\big(f(x)\big)\frac{f(\sigma_\alpha x)-f(x)}{\langle\alpha,x\rangle^2}\\
&&+\frac{\psi\big(f(\sigma_\alpha x)\big)-\psi\big(f(x)\big)-\psi'\big(f(x)\big)[f(\sigma_\alpha x)-f(x)]}{\langle\alpha,x\rangle^2}
\bigg)\\
&=&\psi'\big(f(x)\big)\Delta_\kappa f(x)+\psi''\big(f(x)\big)|\nabla f|^2(x)+2\sum_{\alpha\in\mathcal{R}_+}\kappa_\alpha\frac{\pi_\psi\big(f(\sigma_\alpha x),f(x)\big)}{\langle\alpha,x\rangle^2},
\end{eqnarray*}
where we applied the chain rule for the Laplacian $\Delta$ in the second equality.
\end{proof}

Now we start to prove Theorem \ref{main-thm-2}.
\begin{proof}[Proof of Theorem \ref{main-thm-2}]
From the assumption,  it is clear that \eqref{thm-2-1} implies \eqref{thm-2-0}.

Let $u\in\mathcal{S}(L_V)$ and denote also $u_t=u(t,\cdot)$ and $V_t=V(t,\cdot)$. Applying Lemma \ref{chain-rule} with $\psi(t)=\log t$ for every $t>0$, we obtain
\begin{eqnarray}\label{main-thm-2-1}
L_V\log u_t&=&\frac{\Delta_\kappa u_t}{u_t}-|\nabla\log u_t|^2+\Pi_{\log}(u_t)-V_t\log u_t.
\end{eqnarray}
By \eqref{main-thm-2-1},
\begin{eqnarray*}
\frac{|\nabla u(t,\cdot)|^2(x)}{u(t,x)^2}-\frac{\partial_t u(t,x)}{u(t,x)}&=&\Pi_{\log}\big(u(t,\cdot)\big)(x)-\Delta_\kappa\big(\log u(t,\cdot)\big)(x)+V(t,x)\\
&=&2\sum_{\alpha\in\mathcal{R}_+}\frac{\kappa_\alpha}{\langle\alpha,x\rangle^2}\left(\log\frac{u(t,\sigma_\alpha x)}{u(t,x)} - \frac{u(t,\sigma_\alpha x)}{u(t,x)} +1\right)\\
&&-\Delta_\kappa\big(\log u(t,\cdot)\big)(x)+V(t,x).
\end{eqnarray*}
Let $\xi(t):=\log t -t+1$, $t>0$. It is easy to see that $\xi(t)\leq \xi(1)=0$ for all $t>0$, which implies that $\Pi_{\log}\big(u(t,\cdot)\big)(x)\leq 0$ for all $t>0$ and $x\in\R^d$. Combing this with \eqref{thm-2-1} together, we prove \eqref{thm-2-1+}.

So, it remains to prove that \eqref{thm-2-0} implies \eqref{thm-2-1}. Let $f=u(0,\cdot)$. Since $u(t,x)=\int_{\R^d}f(y)h^V_t(x,y)\,\mu_\kappa(\d y)$, by \eqref{main-thm-2-1} and the assumption (\textbf{H}), we have
\begin{equation}\begin{split}\label{proof-thm-2-1}
&\partial_t u(t,x)+\beta(t,x)u(t,x)\\
&=\int_{\R^d}f(y)\partial_t h^V_t(x,y)\,\mu_\kappa(\d y) + \beta(t,x)\int_{\R^d}f(y)h^V_t(x,y)\,\mu_\kappa(\d y)\\
&=\int_{\R^d}f(y)\big[L_V h^V_t(\cdot,y)(x)+\beta(t,x) h^V_t(x,y)\big]\,\mu_\kappa(\d y)\\
&=\int_{\R^d}f(y)\big[h^V_t(x,y)\Delta_\kappa \log h^V_t(x,y)+h^V_t(x,y)|\nabla\log h^V_t(\cdot,y)(x)|^2\\
&\quad-V(t,x)h_t^V(x,y)-h^V_t(x,y)\Pi_{\log}\big(h^V_t(\cdot,y)\big)(y)+\beta(t,x)h^V_t(x,y)\big]\,\mu_\kappa(\d y)\\
&\geq\int_{\R^d}f(y)\big[h^V_t(x,y)|\nabla\log h^V_t(\cdot,y)(x)|^2-h^V_t(x,y)\Pi_{\log}\big(h^V_t(\cdot,y)\big)(x)\\
&\quad-V(t,x)h_t^V(x,y)\big]\,\mu_\kappa(\d y),
\end{split}\end{equation}
where \eqref{thm-2-0} was applied in the last inequality.

It is easy to see that, by the Cauchy--Schwarz inequality,
\begin{eqnarray*}
|\nabla u(t,\cdot)(x)|^2&\leq&\left(\int_{\R^d}|\nabla h^V_t(\cdot,y)(x)|f(y)\,\mu_\kappa(\d y)\right)^2\\
&\leq&\left(\int_{\R^d}\frac{|\nabla h^V_t(\cdot,y)(x)|^2}{h^V_t(x,y)}f(y)\,\mu_\kappa(\d y)\right)\left(\int_{\R^d}h^V_t(x,y)f(y)\,\mu_\kappa(\d y)\right),
\end{eqnarray*}
which implies that
\begin{eqnarray}\label{proof-thm-2-2}
\int_{\R^d}f(y) |\nabla\log h^V_t(\cdot,y)(x)|^2h^V_t(x,y)\,\mu_\kappa(\d y) \geq u(t,x) |\nabla \log u(t,\cdot)(x)|^2.
\end{eqnarray}

Now let
$${\rm I}=-\int_{\R^d}\Pi_{\log}\big(h^V_t(\cdot,y)\big)(x)h^V_t(x,y)f(y)\,\mu_\kappa(\d y),$$
and
$${\rm II}=-u(t,x)\Pi_{\log}\big(u(t,\cdot)\big)(x).$$
Then
\begin{eqnarray*}
{\rm I}-{\rm II}&=&2\sum_{\alpha\in\mathcal{R}_+}\frac{\kappa_\alpha}{\langle \alpha,x\rangle^2}\int_{\R^d} \left[-\log\frac{h^V_t(\sigma_\alpha x,y)}{h^V_t(x,y)}+\frac{h^V_t(\sigma_\alpha x,y)}{h^V_t(x,y)}-1 \right]h^V_t(x,y)f(y)\,\mu_\kappa(\d y)\\
&&- 2\sum_{\alpha\in\mathcal{R}_+}\frac{\kappa_\alpha}{\langle \alpha,x\rangle^2}\int_{\R^d} \left[-\log\frac{u(t,\sigma_\alpha x)}{u(t,x)}+\frac{u(t,\sigma_\alpha x)}{u(t,x)}-1 \right]h^V_t(x,y)f(y)\,\mu_\kappa(\d y)\\
&=&2\sum_{\alpha\in\mathcal{R}_+}\frac{\kappa_\alpha}{\langle \alpha,x\rangle^2}\int_{\R^d} \left[\eta\Big(\frac{h^V_t(\sigma_\alpha x,y)}{h^V_t(x,y)}\Big) -  \eta\Big(\frac{u(t,\sigma_\alpha x)}{u(t,x)}\Big) \right]h^V_t(x,y)f(y)\,\mu_\kappa(\d y),
\end{eqnarray*}
where $\eta(t):=t-\log t-1$, $t\in\R$. It is clear that $(0,\infty)\ni t\mapsto \eta(t)$ is convex and $ \frac{\d}{\d t}\eta(t)=(t-1)/t$, which implies
$$ \eta(t)- \eta(s)\geq\frac{s-1}{s}(t-s), \quad s,t>0.$$
Hence
\begin{eqnarray*}
&&{\rm I}-{\rm II}\\
&\geq& 2\sum_{\alpha\in\mathcal{R}_+}\frac{\kappa_\alpha}{\langle \alpha,x\rangle^2}
\int_{\R^d} \frac{u(t,\sigma_\alpha x)-u(t,x)}{u(t,\sigma_\alpha x)}\left[\frac{h^V_t(\sigma_\alpha x,y)}{h^V_t(x,y)}-\frac{u(t,\sigma_\alpha x)}{u(t,x)}\right]h^V_t(x,y)f(y)\,\mu_\kappa(\d y)\\
&=&2\sum_{\alpha\in\mathcal{R}_+}\frac{\kappa_\alpha}{\langle \alpha,x\rangle^2}\bigg(
\int_{\R^d} \frac{u(t,\sigma_\alpha x)-u(t,x)}{u(t,\sigma_\alpha x)}h^V_t(r_\alpha x,y)f(y)\,\mu_\kappa(\d y)\\
&&-\int_{\R^d}\frac{u(t,\sigma_\alpha x)-u(t,x)}{u(t,x)}h^V_t(x,y)
f(y)\,\mu_\kappa(\d y)
\bigg)\\
&=&0,
\end{eqnarray*}
which means that
\begin{eqnarray}\label{proof-thm-2-3}
\int_{\R^d}\Pi_{\log}\big(h^V_t(\cdot,y)\big)(x)h^V_t(x,y)f(y)\,\mu_\kappa(\d y)\leq u(t,x)\Pi_{\log}\big( u(t,\cdot)\big)(x).
\end{eqnarray}

Combining \eqref{proof-thm-2-1}, \eqref{proof-thm-2-2} and \eqref{proof-thm-2-3} together, we have
\begin{eqnarray*}\label{proof-thm-2-4}
&&L_V u(t,\cdot)(x)+\beta(t,x)u(t,x)\\
&\geq& u(t,x)|\nabla\log u(t,\cdot)|^2(x)-u(t,x)\Pi_{\log}\big( u(t,\cdot)\big)(x)-V(t,x)u(t,x),
\end{eqnarray*}
which implies
$$-\Delta_\kappa\big(\log u(t,\cdot)\big)(x)\leq \beta(t,x)$$
by \eqref{main-thm-2-1}.
\end{proof}

The proof of the parabolic Harnack inequality by using the Li--Yau inequality is standard (see \cite[Section 2]{LiYau86}), which is  presented here for the sake of completeness.
\begin{proof}[Proof of Corollary \ref{harnack}]
By Theorem \ref{main-thm-2} and Corollary \ref{LY-dho}, it is clear that
\begin{equation}\begin{split}\label{pf-harnack-1}
\frac{|\nabla u(t,\cdot)|^2(x)}{u(t,x)^2}-\frac{\partial_t u(t,x)}{u(t,x)}&\leq (d+2\lambda_\kappa)\coth(2t)+|x|^2,\quad t>0,\,x\in\R^d.
\end{split}\end{equation}
Let $0<s<t<T$, $x,y\in\R^d$, and let
$$\gamma_\tau=\big(t+\tau(s-t),y+\tau(x-y)\big),\quad \tau\in[0,1],$$
be the straight line from $(t,y)$ to $(s,x)$.
Consider the function
$$\phi(\tau):=\log u(\gamma_{\tau}),\quad \tau\in[0,1].$$
 Then
\begin{eqnarray*}
&&\log \frac{u(s,x)}{u(t,y)}=\phi(1)-\phi(0)=\int_0^1 \phi'(\tau)\,\d\tau\\
&=&\int_0^1\left[\Big\langle\frac{\nabla_x u(\gamma_\tau)}{u(\gamma_\tau)}, x-y\Big\rangle-(t-s)\frac{\partial_t u(\gamma_\tau)}{u(\gamma_\tau)}\right]\,\d\tau\\
&\leq&\int_0^1|x-y|\frac{|\nabla_x u(\gamma_\tau)|}{u(\gamma_\tau)}\,\d\tau - (t-s)\int_0^1 \frac{|\nabla_x u(\gamma_\tau)|^2}{u(\gamma_\tau)^2}\,\d\tau\\
&& + (t-s)\int_0^1\left[(d+2\lambda_\kappa)\coth[2(t+\tau(s-t))]+|y+\tau(x-y)|^2\right]\,\d\tau\\
&\leq&|x-y|\Big(\int_0^1\frac{|\nabla_x u(\gamma_\tau)|^2}{u(\gamma_\tau)^2}\,\d\tau\Big)^{1/2}- (t-s)\int_0^1 \frac{|\nabla_x u(\gamma_\tau)|^2}{u(\gamma_\tau)^2}\,\d\tau \\
&&+(d+2\lambda_\kappa)\int_s^t\coth(2\tau)\,\d\tau+(t-s)\int_0^1|y+\tau(x-y)|^2\,\d \tau\\
&\leq& \frac{|x-y|^2}{4(t-s)}+(d+2\lambda_\kappa)\int_s^t\coth(2\tau)\,\d\tau+(t-s)\int_0^1|y+\tau(x-y)|^2\,\d \tau,
\end{eqnarray*}
where we applied \eqref{pf-harnack-1} in the first inequality, the Cauchy--Schwarz inequality in the second one and Young's inequality in the last one. Thus,
$$u(s,x)\leq u(t,y)\Big(\frac{\sinh(2t)}{\sinh(2s)}\Big)^{\frac{d+2\lambda_\kappa}{2}}
\exp\left(\frac{|x-y|^2}{4(t-s)}+(t-s)\frac{|x|^2+|y|^2+\langle x,y\rangle}{3}\right).$$

Moreover, by the elementary inequality, i.e.,
$$\frac{\sinh(t)}{\sinh(s)}=\frac{e^t-e^{-t}}{e^s-e^{-s}}\leq e^{t-s}\Big(\frac{t}{s}\Big),\quad t\geq s>0,$$
we have
$$u(s,x)\leq u(t,y)\Big(\frac{t}{s}\Big)^{\frac{d+2\lambda_\kappa}{2}}\exp\left(\frac{|x-y|^2}{4(t-s)}+(t-s)\Big[d+2\lambda_\kappa+\frac{|x|^2+|y|^2+\langle x,y\rangle}{3}\Big]\right).$$
\end{proof}

Finally, we turn to prove the remark.
\begin{proof}[Proof of Remark \ref{rmk}] We divide the proof into two parts.

(1) For any $t>0,\,x,y\in\R^d$ with $x=(x_1,\cdots,x_d)$ and $y=(y_1,\cdots,y_d)$, let $q_t(x,y)=p_t(x,y)/\varrho_t(x)$.  By \eqref{hc-kernel-2} and the fact (see e.g. \cite[page 2366]{ADH2019}) that
$$\int_1^1(1-s)^{\kappa_j-1}(1+s)^{\kappa_j}\,\d s=\frac{\sqrt{\pi}\Gamma(\kappa_j)}{\Gamma(\kappa_j+1/2)},\quad j=1,\cdots,d,$$
we have
$$q_t(x,y)=\exp\left(-\frac{|y|^2}{4t}\right)\prod_{j=1}^d \frac{\Gamma(\kappa_j+1/2)}{\sqrt{\pi}\Gamma(\kappa_j)}\int_{-1}^{1}g_j(s)\exp\left(\frac{x_jy_js}{\sinh(2t)}\right)\,\d s,$$
where $g_j(s):=(1-s)^{\kappa_j-1}(1+s)^{\kappa_j}$, $j=1,\cdots,d$. Then it is easy to see that
$$\partial_{x_i x_j}^2\log q_t(x,y)=0,\quad 1\leq i\neq j\leq d, $$
and for each $j=1,\cdots,d$,
\begin{eqnarray*}
&&\partial_{x_j x_j}^2\log q_t(x,y)\\
&=&\frac{y_j^2}{\sinh(2t)^2}\left[\frac{\int_{-1}^1 s^2g_j(s)\exp\left(\frac{x_jy_js}{\sinh(2t)}\right)\,\d s}{\int_{-1}^1 g_j(s)\exp\left(\frac{x_jy_js}{\sinh(2t)}\right)\,\d s}
-\frac{\left(\int_{-1}^1 s g_j(s)\exp\left(\frac{x_jy_js}{\sinh(2t)}\right)\,\d s\right)^2}{\left(\int_{-1}^1 g_j(s)\exp\left(\frac{x_jy_js}{\sinh(2t)}\right)\,\d s\right)^2}\right]\\
&\geq&0,
\end{eqnarray*}
by the Cauchy--Schwarz inequality. Hence $x\mapsto \log q_t(x,y)$ is convex. Thus, for every $\theta\in(0,1)$ and every $x_1,x_2\in\R^d$, by H\"{o}lder's inequality, we deduce that
\begin{eqnarray*}
&&\Big(\frac{u_t}{\varrho_t}\Big)\big(\theta x_1+(1-\theta)x_2\big)\\
&=&\int_{\R^d}q_t\big(\theta x_1+(1-\theta)x_2,y\big)u(0,y)\,\mu_\kappa(\d y)\\
&\leq&\int_{\R^d} q_t(x_1,y)^\theta q_t(x_2,y)^{1-\theta}u(0,y)\,\mu_\kappa(\d y)\\
&\leq&\left(\int_{\R^d} q_t(x_1,y)u(0,y)\,\mu_\kappa(\d y)\right)^\theta\left(\int_{\R^d} q_t(x_2,y)u(0,y)\,\mu_\kappa(\d y)\right)^{1-\theta}\\
&=&\left[\Big(\frac{u_t}{\varrho_t}\Big)(x_1)\right]^\theta\left[\Big(\frac{u_t}{\varrho_t}\Big)(x_2)\right]^{1-\theta},
\end{eqnarray*}
which implies that $\log(\frac{u_t}{\varrho_t})(\cdot)$ is convex.

(2) It follows from the similar proof of Corollary \ref{LY-dho} that \eqref{Rem-LY} implies \eqref{Rem-Har}.  We prove the converse part.
Let $s>0$ and $x\in \R^d$ be fixed. For any $z\in \R^d$, let
$$\gamma_{\varepsilon}:=(s+\varepsilon, x+\varepsilon z),\quad \varepsilon>0.$$
Applying \eqref{Rem-Har} with $t=s+\varepsilon$ and $y=x+\varepsilon z$, we have
\begin{equation}\begin{split}\label{rmk-1}
u(s,x)&\le u(s+\varepsilon,x+\varepsilon z)\exp\left\{ \frac{\varepsilon|z|^2}{4}+\varepsilon\int_0^1\beta\big(s+(1-\tau)\varepsilon, x+(1-\tau)\varepsilon z\big)\,\d\tau\right\}\\
&=:u(s+\varepsilon,x+\varepsilon z) e^{\Psi_{x,z}(\varepsilon)}.
\end{split}\end{equation}
As $\varepsilon\to 0$, it is easy to see that $\Psi_{x,z}(\varepsilon)$ can be approximated by
$$%A:=\frac{\varepsilon|\omega|^2}{4}+\varepsilon\int_0^1\beta(s+(1-\tau)\varepsilon, x+(1-\tau)\varepsilon\omega)d\tau\sim
\frac{\varepsilon|z|^2}{4}+\varepsilon\beta(s,x).$$
By \eqref{rmk-1}, a first-order Taylor expansion for the function $\varepsilon\mapsto u(s+\varepsilon,x+\varepsilon z) e^{\Psi_{x,z}(\varepsilon)}$ around $0$ leads to that
$$
0\le \langle\nabla_x  u(s,x), z\rangle  +\partial_su(s,x)+u(s,x)\left(\frac{|z|^2}{4}+\beta(s,x)\right).
$$
Taking $z=-2 \frac{\nabla_x u}{u}(s,x)$, we immediately obtain
$$
\frac{|\nabla_x u|^2}{u^2}(s,x)-\frac{\partial_s u}{u}(s,x)\le \beta(s,x),
$$
which is  \eqref{Rem-LY}.
\end{proof}

\subsection*{Acknowledgment}\hskip\parindent
The authors would like to express their sincere thanks to the anonymous referee for his/her  careful reading and  valuable suggestion.
The first named author would like to thank Dr. Niushan Gao for helpful discussions and acknowledge the Department of Mathematics and the Faculty of Science at Ryerson University for financial support and the financial support from the National Natural Science Foundation of China (Grant No. 11831014). The second named author would like to acknowledge the financial support from Qing Lan Project of Jiangsu.

%\subsection*{Data availability statements}\hskip\parindent
%No datasets were generated or analysed during the current study.

\end{document}